\documentclass{amsart}
\newcommand{\notpaper}[1]{}

\usepackage[left=2.5cm,right=1cm,
    top=2cm,bottom=2cm,bindingoffset=0cm]{geometry}

\usepackage[T2A]{fontenc}
\usepackage[utf8]{inputenc}
\usepackage[russian]{babel}
\usepackage{amssymb}
\usepackage{a4wide}
\usepackage{amsfonts}

\makeatletter
\def\@settitle{\begin{center}%
    \baselineskip14\p@\relax
    \bfseries
    \@title
  \end{center}%
}
\makeatother

\usepackage{todonotes} % for comments and "to do" notes
\usepackage{amsfonts, amsmath, amssymb, amsthm, amscd, hyperref}

\usepackage{hyperref}
\hypersetup{
  colorlinks   = true, %Colours links instead of ugly boxes
  urlcolor     = blue, %Colour for external hyperlinks
  linkcolor    = blue, %Colour of internal links
  citecolor   = red %Colour of citations
}

\usepackage{graphicx}

\newtheorem{theorem}{Теорема}

\newtheorem*{theorem*}{Теорема}

\newtheorem*{proposition*}{Предложение}
\newtheorem{lemma}[theorem]{Лемма}
\newtheorem*{lemma*}{Лемма}
\newtheorem{problem}[theorem]{Проблема}
\newtheorem*{problem*}{Проблема}
\newtheorem{claim}[theorem]{Утверждение}
\newtheorem*{claim*}{Утверждение}

\newtheorem*{conjecture*}{Гипотеза}
\newtheorem{corollary}[theorem]{Следствие}
\newtheorem*{corollary*}{Следствие}
\newtheorem{comment}[theorem]{Замечание}
\newtheorem*{comment*}{Замечание}

\theoremstyle{definition}

\theoremstyle{definition}
\newtheorem*{definition*}{Определение}

\newenvironment{lemmaproof}
{\par\noindent{\bf Доказательство. %Леммы \arabic{section}.\arabic{lemma}
}}
{\hfill$\scriptstyle\blacksquare$}

\usepackage{setspace}
\author{Бикеев Артур Игоревич}
\thanks{$^*$ При поддержке РФФИ, грант № 19-01-00169. Выражаю благодарность своему научному руководителю А. Скопенкову, а также А. Воропаеву, В. Губареву, Е. Когану, А. Медных, И. Медных и В. Ретинскому
за полезные замечания по статье.}

\begin{document}

\title{Реализуемость дисков с ленточками на ленте Мёбиуса$^*$}

\maketitle

\begin{abstract}
% Для симметричной матрицы с элементами из $\Z_2$ можно поменять некоторые элементы на главной диагонали так, чтобы все ненулевые строки матрицы оказались равны, тогда и только тогда, когда нельзя сделать такую одинаковую перестановку строк и столбцов\footnote{Это означает, что строки и столбы последовательно занумерованы числами от $1$ до $n$ и выбрана перестановка~$f$ этих чисел; строки и столбцы переставляются так, что $i,j$-ый элемент становится $f(i),f(j)$-ым.}, что в верхнем левом углу полученной матрицы будет стоять подматрица вида:

Назовём \textit{иероглифом} циклическое 
слово длины $2n$ из $n$ букв, в котором каждая буква 
встречается дважды (стандартный термин: мультиграф с вращениями, имеющий одну вершину). 
Возьмём выпуклый многоугольник на плоскости.
Отметим на ограничивающей его ломаной $2n$ непересекающихся отрезков и обозначим их буквами из слова в том порядке, в каком эти буквы следуют в слове. 
Для каждой буквы соединим соответствующие два отрезка ленточкой так, чтобы различные ленточки не пересекались.
Любой полученный таким образом объект будем называть \textit{диском с ленточками, соответствующим данному иероглифу}.
Назовем иероглиф \textit{слабо реализуемым} на ленте Мёбиуса,  если из неё можно вырезать некоторый диск с ленточками, соответствующий данному иероглифу.
В работе приводится критерий слабой реализуемости, дающий квадратичный (по количеству букв) алгоритм. Известные критерии, основанные на формуле Эйлера и теореме Мохара, дают экспоненциальные алгоритмы.
% Из
Приведённый критерий также основан на критерии Мохара реализуемости диска с ленточками на ленте Мёбиуса.
%При помощи теоремы Мохара слабая реализуемость иероглифа на ленте Мёбиуса приводится к задаче из линейной алгебры.
%Приводим критерий слабой реализуемости, дающий квадратичный (по кол-ву букв) алгоритм. Известные критерии, основанные на формулах Эйлера и Мохара, дают экспоненциальные алгоритмы.
\end{abstract}

\maketitle

\section*{Введение и основные результаты}

%К следующей линейно-алгебраической лемме с помощью теоремы Мохара сводится топологическая задача о вложимости графа в поверхность:

В работе строится квадратичный алгоритм распознавания вложимости графа с дополнительной структурой в ленту Мёбиуса. История вопроса излагается в замечании \ref{c:hist}. Основным результатом работы является теорема \ref{t:main}. Следующая алгебраическая лемма тесно с ней связана.

\begin{lemma}\label{l:main}
Пусть $M$ - симметричная матрица, элементами которой являются нули и единицы. Тогда следующие условия эквивалентны:

\begin{enumerate}
    %\item Изменением некоторых элементов на главной диагонали можно из матрицы $M$ получить матрицу, ранг которой не превосходит $1$.
    \item Можно сделать такую одинаковую перестановку строк и столбцов\footnote{Это означает, что строки и столбы последовательно занумерованы числами от $1$ до $n$ и выбрана перестановка~$f$ этих чисел; строки и столбцы переставляются так, что $i,j$-ый элемент становится $f(i),f(j)$-ым.} матрицы $M$ и изменить некоторые элементы на главной диагонали таким образом, что в верхнем левом углу полученной матрицы будет стоять подматрица, заполненная единицами, вне которой стоят нули.
    %\item Одинаковой перестановкой столбцов и строк\footnote{Это означает, что строки и столбцы занумерованы подряд числами от $1$ до $n$ и некоторая перестановка $\sigma$ чисел от $1$ до $n$ применена
%и к строкам, и к столбцам.}
%и изменением некоторых элементов на главной диагонали можно из матрицы $M$ получить матрицу, у которой некоторый левый верхний квадрат заполнен единицами, а вне его стоят нули.
    \item %У матрицы $M$ нельзя выбрать %одинаковые подмножества строк и 
%подмножество 
%столбцов%\footnote{Это означает, что строки и столбцы занумерованы подряд числами от $1$ до $n$ и выбраны подмножества терок и 
%тех 
%столбцов, номера которых принадлежат некоторому множеству $S$.}, на пересечении которых стоит подматрица вида
Нельзя сделать такую одинаковую перестановку строк и столбцов матрицы $M$, что в верхнем левом углу полученной матрицы будет стоять подматрица вида
    $$P = \begin{pmatrix}
    *& 1& 1\\
    1& *& 0\\
    1& 0& *
\end{pmatrix}
%%,
{\rm \text{ или } }
Q = \begin{pmatrix}
    *& 1& 0& 0\\
    1& *& 0& 0\\
    0& 0& *& 1\\
    0& 0& 1& *\\
\end{pmatrix},
$$ где через $*$ обозначены произвольные (возможно, различные) элементы.% у которой вместо $*$ стоят нули и единицы.
\footnote{Условиям (1)-(2) эквивалентно следующее условие: изменением некоторых элементов на главной диагонали можно из матрицы $M$ получить матрицу, ранг которой не превосходит $1$. Равносильность этого условия условию (1) очевидна.}
\end{enumerate}
\end{lemma}

%Импликации $(1) \Longleftrightarrow (2)$ леммы доказываются непосредственно.

Импликация $(1) \Longrightarrow (2)$ легко следует из того, что при любой расстановке нулей и единиц на главной диагонали каждой из матриц $P$ и $Q$ верхняя строка и нижняя строка полученной матрицы будут ненулевыми и различными.

%Импликация $(2) \Longrightarrow (3)$ доказывается следующим образом. Из $(2)$ непосредственно следует выполнение следующего свойства: все ненулевые строки матрицы $M$ равны. Но две верхние строки матрицы $P$ (аналогично, матрицы $Q$) не равны между собой и не являются нулевыми при любой замене $*$ на нули и единицы.

Импликация $(2) \Longrightarrow (1)$  леммы \ref{l:main} фактически доказана при доказательстве импликации $(4) \Longrightarrow (3)$ теоремы \ref{t:main},  см. замечание \ref{r:tl}.

\bigskip
Назовем \textbf{иероглифом} циклическое 
слово длины $2n$ из 
%%$n$ букв, в котором каждая буква встречается дважды с точностью до переименования букв 
$n$ различных букв, в котором каждая буква встречается дважды 
(стандартные термины: хордовая диаграмма или мультиграф с вращениями, имеющий одну вершину).
Возьмём выпуклый многоугольник на плоскости.
Отметим на ограничивающей его ломаной непересекающиеся отрезки и
обозначим их буквами из слова в том порядке, в каком эти буквы следуют в слове. 
Для каждой буквы соединим соответствующие два отрезка ленточкой
(не обязательно в плоскости) так, чтобы
различные ленточки не пересекались. 
(Ленточки могут быть и перекрученные, и нет.)
Любой из полученных таким образом объектов будем называть \textbf{диском с ленточками}, соответствующим данному иероглифу, см. рис. \ref{f:ababcdcd_1}.

\begin{figure}[ht]
    \center{\includegraphics[scale=0.5, width=250pt]{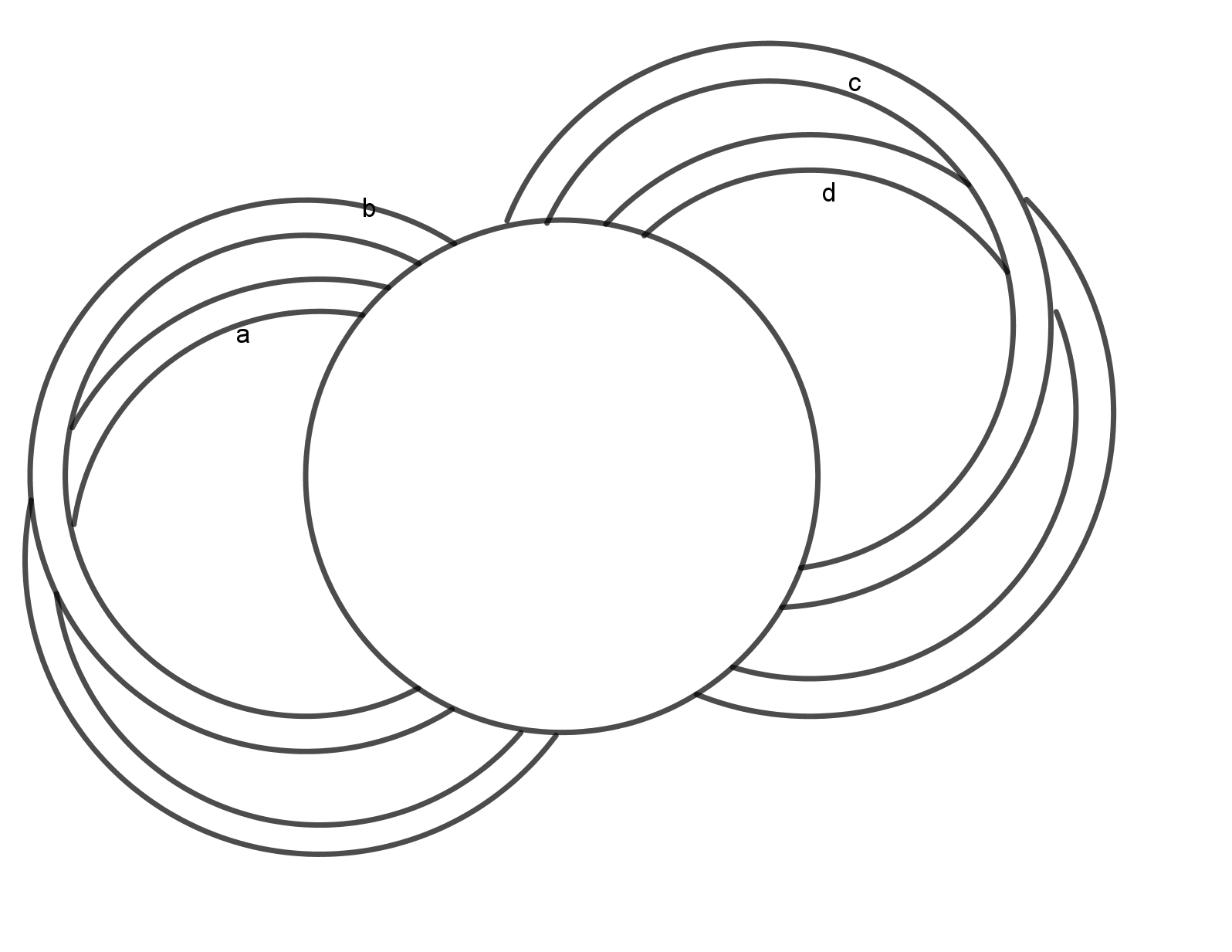}}
    \caption{Пример диска с ленточками, соответствующего иероглифу $ababcdcd$}
\label{f:ababcdcd_1}
\end{figure}

Назовем иероглиф \textbf{слабо реализуемым} на ленте Мёбиуса, если из неё можно вырезать некоторый диск с ленточками, соответствующий данному иероглифу (ср. с замечанием \ref{c:hist}).

\begin{theorem}[доказана далее]\label{t:alg}
Существует квадратичный алгоритм проверки слабой реализуемости иероглифа на ленте Мёбиуса.
\end{theorem}

%\newpage

Будем говорить, что две различные буквы $a$ и $b$ иероглифа \textbf{перекрещиваются}, если они идут в этом иероглифе в чередующемся порядке ($abab$, а не $aabb$).

Иероглифы считаются равными, если получаются один из другого взаимно однозначной заменой букв или осевой симметрией.

\begin{theorem}\label{t:main}
Пусть $H$ --- иероглиф. Следующие условия эквивалентны:
\begin{enumerate}
    \item 
Иероглиф $H$
слабо реализуем на ленте Мёбиуса.
    \item Буквы иероглифа $H$ можно раскрасить в красный и синий цвета так, что любые две красные буквы перекрещиваются, а никакая синяя буква не перекрещивается ни с какой другой буквой.
    \item 
%%$H$ операциями 
Операциями
удаления пар одинаковых букв, таких, что каждая из этих букв не перекрещивается ни с какой другой буквой иероглифа $H$, можно 
%%свести 
свести $H$
к 
иероглифу вида
$a_1 a_2 ... a_m a_1 a_2 ...a_m$ (возможно, $m=0$).
    \item Из $H$ операциями удаления пар одинаковых букв нельзя получить иероглифы $abcacb$ и $ababcdcd$.
    \footnote{Условиям (1)-(4) эквивалентно следующее условие: найдется такая матрица над $\mathbb{Z}_2$ ранга не более $1$ размера $n\times n$, где $n$ - число различных букв 
в иероглифе
$H$, что в 
клетках
вне главной диагонали матрицы стоит $0$, если соответствующие буквы иероглифа $H$ не перекрещиваются, и $1$ в противном случае. Равносильность этого условия и условия (1) является частным случаем теоремы \ref{t:mohar} Мохара, сформулированной далее.}
\end{enumerate}
\end{theorem}

В теореме \ref{t:main} импликации  (3) $\Longleftrightarrow$ (2)$\Longrightarrow$ (1) очевидны. Импликации (1) $\Longrightarrow$ (4) $\Longrightarrow$ (3) доказаны далее.

\begin{comment}[\textbf{история проблемы}]\label{c:hist}
Базовые ссылки по данной теме: \cite{MT01},{\cite[$\S 2$]{Sk20}}, \cite{LZ}. По поводу обобщения на произвольные графы см. \cite{MT01},{\cite[$\S 2$]{Sk20}}, \cite{LZ} и \cite{Ko20}
. По поводу связи с интегрируемыми гамильтоновыми системами см. \cite{BFM90}. 
%По поводу других задач, связанных с лентой Мёбиуса, см. \cite{F}.

Известны полиномиальные алгоритмы распознавания вырезаемости конкретного двумерного многообразия из ленты Мёбиуса, например, использующие эйлерову характеристику или теорему \ref{t:mohar} Мохара.

%Более сложный полиномиальный алгоритм, для любого $m$ распознающий вырезаемость диска с $n$ ленточками из диска с $m$ лентами Мёбиуса, фактически построен в \cite{Ko20}.

Однако из существования этих алгоритмов не следует существование полиномиального алгоритма распознавания слабой реализуемости иероглифа на ленте Мёбиуса. Действительно, каждому иероглифу с $n$ парами букв соответствует не одно, а $2^n$ двумерных многообразий (дисков с ленточками), так как каждая ленточка может быть либо перекрученной, либо нет. Слабая реализуемость иероглифа на ленте Мёбиуса эквивалентна вырезаемости из ленты Мёбиуса хотя бы одного из них. %Таким образом, теорема Мохара дает экспоненциальный алгоритм проверки слабой реализуемости иероглифа на ленте Мёбиуса, основанный на полном переборе, а теорема \ref{t:main} позволяет проверить это, не рассматривая каждое из многообразий в отдельности.
\end{comment}

\section*{Доказательства и нерешённая проблема}

% Теорема 2.8.11с для общего случая,
% 2.8.8c, 6.7.7 (стр. 169) одну можно вырезать из другой, когда ...
%!!!!!!!!!!!! Утверждение.

\begin{proof}[\textbf{Доказательство импликации (4) $\Longrightarrow$ (3) теоремы \ref{t:main}}]
    Пусть для иероглифа $H$ выполняется условие (4). Обозначим через $H_1$ иероглиф, получаемый из иероглифа $H$ удалением всех таких букв, которые не перекрещиваются ни с одной другой. Условие (4) выполняется и для иероглифа $H_1$. Предположим, что в иероглифе $H_1$ есть пара не перекрещивающихся букв $a$ и $c$.
    
    Если найдется буква $b$, перекрещивающаяся и с $a$, и с $c$, то, удалив из $H_1$ все буквы, кроме $a$, $b$ и $c$, получим иероглиф $abacbc$, равный иероглифу $abcacb$,  что дает противоречие с условием (4). Значит, любая буква, отличная от $a$ и $c$, перекрещивается не более чем с одной из них.

%%Поэтому 
Следовательно,
в силу того, что 
%%у 
в иероглифе
$H_1$ любая буква перекрещивается хотя бы с одной другой, найдутся две различные буквы $b$ и $d$ такие, что $b$ перекрещивается с $a$, но не перекрещивается с $c$, а $d$ перекрещивается с $c$, но не перекрещивается с $a$. Поэтому если буквы $b$ и $d$ не перекрещиваются, то, удалив все буквы, кроме $a$, $b$, $c$ и $d$, получим иероглиф $ababcdcd$. Если же $b$ и $d$ перекрещиваются, то, удалив все буквы, кроме $a$, $b$ и $d$, получим иероглиф $badbda$, равный иероглифу $abcacb$. Во всех случаях получили противоречие с тем, что $a$ и $c$ не перекрещиваются. Значит, любые две буквы иероглифа $H_1$ перекрещиваются, 
%%тогда 
и тогда
он имеет вид $a_1 a_2 ... a_n a_1 a_2 ... a_n$. Следовательно, иероглиф $H$ удовлетворяет условию (3).

\end{proof}

\textbf{Матрицей перекрещиваний} диска с $n$ ленточками $D$ называется матрица размера $n \times n$, у которой

$\bullet$ на главной диагонали стоит $1$, если соответствующая ленточка перекручена, и $0$ 
%%иначе, 
в противном случае,
и 

$\bullet$ вне главной диагонали на пересечении строки $i$ и столбца $j$ стоит $0$, если соответствующие ленточки 
%диска с ленточками $D$ 
не перекрещиваются, и $1$ в противном случае.

Например, матрицей перекрещиваний любого диска с ленточками, соответствующего иероглифу $abcacb$, является матрица вида $P$ (см. лемму \ref{l:main}), в которой элементы на главной диагонали зависят от перекрученности ленточек.
\notpaper{
$$\begin{pmatrix}
*& 1& 1\\
1& *& 0\\
1& 0& *
\end{pmatrix}$$
(элементы на главной диагонали зависят от перекрученности ленточек.)}

\begin{figure}[h]
\center{\includegraphics[scale=0.5, width=200pt]{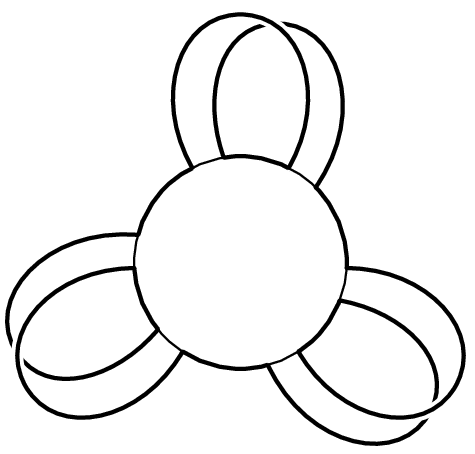}}
\caption{Диск с тремя лентами Мёбиуса}
\label{f:disk}
\end{figure}

\textbf{Диском с $m$ лентами Мёбиуса} (см. рис. \ref{f:disk}) называется объединение круга и $m$  ленточек, 
%%при 
в
котором

$\bullet$ каждая ленточка приклеивается двумя отрезками к граничной окружности $S$ 
%%круга и 
круга, а
направления на этих
отрезках, задаваемые произвольным направлением на~$S$, <<сонаправлены вдоль ленточки>>,

$\bullet$ ленточки 
%%<<отделенные>>, 
<<отделены>>,
т. е. приклеены к $2m$ попарно не пересекающимся отрезкам на~$S$.

\textbf{Рангом над $\mathbb{Z}_2$} матрицы из нулей и единиц будем называть размерность пространства ее строк над полем $\mathbb{Z}_2$.

 Следующая теорема была сформулирована и применена в \cite{Mo89}.

\begin{theorem}\label{t:mohar}
Диск с ленточками можно вырезать из диска с $m$ лентами Мёбиуса тогда и только тогда, когда ранг над $\mathbb{Z}_2$ матрицы перекрещиваний этого диска с ленточками не превосходит $m$.
\end{theorem}
% проверить утверждение 2.8.8c
% 2.8.11?

 Теорему \ref{t:mohar} можно доказать, например, несложным применением формы пересечений двумерного многообразия.  См. также {\cite[утв. 2.8.8(с)]{Sk20}} и {\cite[утв. 6.7.7]{Sk20}}.

\begin{proof}[\textbf{Доказательство импликации (1) $\Longrightarrow$ (4) теоремы \ref{t:main}}]
При любой расстановке нулей и единиц на главной диагонали каждой из матриц $P$ и $Q$ из леммы \ref{l:main} две верхние строки полученной матрицы будут ненулевыми и различными.
Матрицы $P$ и $Q$ являются матрицами перекрещиваний дисков с ленточками, соответствующих иероглифам $abcacb$ и $ababcdcd$. Поэтому, по теореме Мохара, иероглифы $abcacb$ и $ababcdcd$ не являются слабо реализуемыми на ленте Мёбиуса.
\end{proof}

\begin{comment}[\textbf{связь леммы \ref{l:main} и теоремы \ref{t:main}}]\label{r:tl}

Теорема \ref{t:main} является топологической версией леммы \ref{l:main}. При этом условие (1) леммы аналогом условия (2) теоремы, условие (2) леммы является аналогом условия (3) теоремы. Матрицы перекрещиваний любых дисков с ленточками, соответствующих иероглифу $H$ из теоремы, получаются друг из друга одинаковыми перестановками столбцов и изменением некоторых элементов на главной диагонали. Произвольной из этих матриц можно поставить в соответствие матрицу $M$ из леммы. Тогда применение к иероглифу $H$ нескольких операций удаления одинаковых букв соответствует взятию левого верхнего угла некоторой матрицы, полученной из $M$ одинаковыми перестановками столбцов и изменением некоторых элементов на главной диагонали. Каждая буква иероглифа соответствует симметричным относительно главной диагонали строке и столбцу матрицы $M$. Матрицы $P$ и $Q$ в пункте $(2)$ леммы являются аналогами иероглифов $abcacb$ и $ababcdcd$ в пункте $(4)$ теоремы. Доказательство импликации $(2) \Longrightarrow (1)$  леммы \ref{l:main} аналогично доказательству импликации $(4) \Longrightarrow (3)$ теоремы \ref{t:main}.

%Импликация $(3) \Longrightarrow (2)$  леммы \ref{l:main} фактически доказана при доказательстве импликации $(4) \Longrightarrow (3)$ теоремы \ref{t:main}, потому что взятие единичного левого верхнего квадрата матрицы, полученной из матрицы $M$ в пункте $(2)$ леммы, соответствует операциям удаления пар одинаковых букв в пункте $(3)$ теоремы, а матрицы $P$ и $Q$ в пункте $(3)$ леммы соответствуют иероглифам $abcacb$ и $ababcdcd$ в пункте $(4)$ теоремы.
\end{comment}

\begin{proof}[\textbf{Доказательство теоремы \ref{t:alg}}]
Условие (3) из теоремы \ref{t:main} проверяется за квадратичное (относительно длины иероглифа) время. Докажем это. Построим граф $G$, вершины которого соответствуют буквам иероглифа $H$, 
%%две 
а две
вершины соединены ребром, если соответствующие буквы перекрещиваются. Эта процедура занимает квадратичное относительно длины иероглифа время. Тогда условие (3) эквивалентно тому, что $G$ есть объединение клики и, возможно, нескольких изолированных вершин. Проверка всех вершин на изолированность занимает квадратичное время. После этого проверка того, что все не изолированные вершины образуют клику, также занимает квадратичное время.
\end{proof}

\begin{problem}\label{problem}
Пусть дана матрица $M$ размера $n \times n$ из нулей и единиц, на диагонали которой 
%%нули. 
стоят нули.
Обозначим через $R(M)$ наименьший ранг над $\mathbb Z_2$ матрицы, полученной из $M$ изменением 
%%некоторых 
каких-нибудь
диагональных элементов. Найти быстрый (по $n$) алгоритм, вычисляющий $R(M)$.
\end{problem}

\begin{corollary}[из теоремы \ref{t:mohar}]\label{c:R}
Диск с $n$ ленточками с матрицей перекрещиваний $M$ реализуем на диске с $m$ лентами Мёбиуса тогда и только тогда, когда $R(M) \leq m$. 
\end{corollary}

%\begin{comment*}\label{c:ko}
В статье {\cite{Ko20}} для каждого фиксированного целого неотрицательного $m$ дан полиномиальный (по $n$) алгоритм, проверяющий
%%, верно ли, что 
неравенство
$R(M)\leq m$. Этот и другие близкие результаты приводятся в  \cite{VGDNPS}.
%\end{comment*}

%Проблема \ref{problem} интересна в силу следующего утверждения.

%С другой стороны, доказанный вначале критерий реализуемости иероглифа $H$ на ленте Мёбиуса эквивалентен следующему: можно так расставить нули и единицы на главной диагонали матрицы $M$, получив матрицу $M_1$, и найдется такая ортогональная матрица $S$, что у матрицы $S^T M_1 S$ некоторая квадратная подматрица, примыкающая к левому верхнему углу, состоит из единиц, а вне этой подматрицы все элементы равны нулю.

Следующий факт показывает, что задача о реализуемости иероглифа на ленте Мёбиуса сводится лишь к частному случаю проблемы \ref{problem}.

\begin{claim}\label{c:matrix}
Не любая матрица из нулей и единиц является матрицей перекрещиваний 
%%какого-то 
какого-нибудь
диска с ленточками.
% Пояснения к примеру.
\end{claim}

\begin{proof}[\textbf{Доказательство}]
Рассмотрим матрицу смежности графа $G$, изображенного на рис. \ref{f:matrix}.
\begin{figure}[h]
    \center{\includegraphics[scale=0.5, width=150pt]{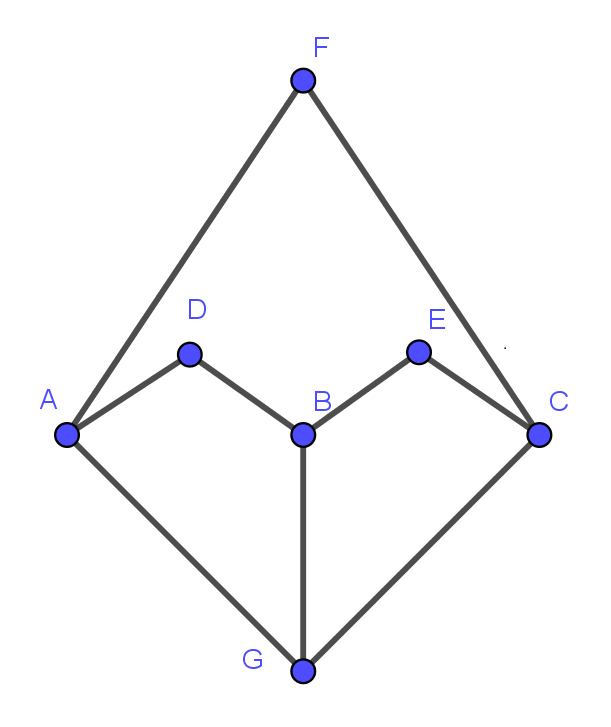}}
    \caption{К утверждению \ref{c:matrix}}
    \label{f:matrix}
\end{figure}
Его вершины $A$, $B$, $C$ попарно не соединены. Если существует диск с ленточками, матрица перекрещиваний которого равна матрице смежности графа $G$, то концы ленточек, соответствующих вершинам $A$,$B$,$C$, расположены в нём в одном из $4$ следующих циклических порядков: $aabbcc, abccba, acbbca, baccab.$ В случае циклического порядка $aabbcc$ ленточка, соответствующая вершине $G$, не может перекрещиваться со всеми тремя ленточками, соответствующими вершинам $A,B,C$. В случае циклического порядка $abccba$ ленточка, соответствующая вершине $F$, не может одновременно перекрещиваться с ленточками, соответствующим вершинам $A$ и $C$ и не перекрещиваться с ленточкой, соответствующей вершине $B$. Аналогично для остальных двух случаев. Получаем противоречие.
\end{proof}

%Не любые два диска с ленточками, у которых совпадают матрицы перекрещиваний, соответствуют одинаковым иероглифам (однако, они обязательно гомеоморфны). Пример - для любого $n > 2$ есть диски с ленточками, соотв. различным иероглифам (например, $aabbcc$ и $abbacc$), в каждом из которых ленточки попарно не перекрещиваются (и матрицы, таким образом, совпадают).

Любые два диска с ленточками, у которых совпадают матрицы перекрещиваний, гомеоморфны, однако, они могут соответствовать разным иероглифам
%%, 
---
например, диски с ленточками, соответствующие иероглифам $aabbcc$ и $abbacc$ соответственно (у каждого из них любые две буквы не перекрещиваются).

\section*{Приложение: прямое доказательство импликации (1) $\Longrightarrow$ (4) теоремы \ref{t:main}}

Иероглифы $abcacb$ и $ababcdcd$ не являются слабо реализуемыми на ленте Мёбиуса в силу теоремы 
%%11 
\ref{t:betty}
Бетти и леммы \ref{l:curve}.

\begin{theorem}\label{t:betty} Объединение любых двух различных 
%%ломаные 
замкнутых ломаных
на ленте Мёбиуса разбивает её.
\end{theorem}

\begin{proof}[\textbf{Доказательство}]

Теорема \ref{t:betty} следует из неравенства Эйлера для ленты Мёбиуса \cite[2.8.2]{Sk20} и теоремы Римана для ленты Мёбиуса \cite[2.8.3.(b)]{Sk20}. Приведем детали.

Объединение двух попарно не пересекающихся ломаных разбивает ленту Мебиуса в силу теоремы \cite[2.8.3.(b)]{Sk20}.

Пусть данные ломаные пересекаются. Можно считать, что вершины и ребра этих ломаных образуют связный граф. 
%Согласно неравенству Эйлера для ленты Мёбиуса, 
Для любого связного графа с $V$ вершинами и $E$ рёбрами, изображённого без самопересечений на ленте Мёбиуса и разбивающего её на $F$ граней, выполнено неравенство Эйлера: $V-E+F\geq 1$. Легко проверить, что в данном случае $V<E$, поэтому $F>1$, значит, эти ломаные разбивают ленту Мёбиуса.

\end{proof}

\begin{lemma}\label{l:curve}
На любом диске с ленточками, соответствующем иероглифу $abcacb$ или $ababcdcd$, найдутся две кривые, пересекающиеся в конечном числе точек, которые не разбивают этот диск с ленточками.
\end{lemma}

%В силу теоремы Римана из этой леммы следует, что иероглифы $abcacb$ и $ababcdcd$ не являются слабо реализуемыми на ленте Мёбиуса.

\begin{lemmaproof}
На рис. \ref{f:abcacb} и рис. \ref{f:ababcdcd} показаны примеры таких пар кривых на дисках с неперекрученными ленточками, соответствующих иероглифам $abcacb$ и $ababcdcd$. Покажем, что если заменить одну или несколько ленточек любого из этих дисков с ленточками на перекрученную, то соответствующие пары кривых по-прежнему не будут разбивать полученные диски с ленточками.

\begin{figure}[h]
    \center{\includegraphics[scale=0.5, width=250pt]{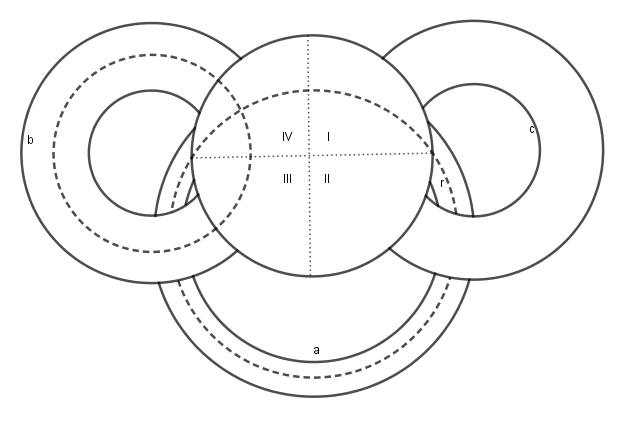}}
    \caption{Диск с ленточками, соответствующий иероглифу $abcacb$, и 
%%двух кривых 
две кривые
на нём.}
\label{f:abcacb}
\end{figure}

%Для иероглифа $abcacb$ и диска с ленточками, изображенного на рис.5: если заменить некоторые ленточки на перекрученные, то ленточка $c$ останется в одной компоненте связности. Обе части ленточки $a$, на которые ее разобьет проходящая по ней кривая, не будут отделены кривыми от ленточки $c$. Значит, ленточки $a$ и $c$ лежат целиком в одной компоненте связности. Каждая из частей ленточки $b$, на которые ее разобьет проходящая по ней кривая, лежит в одной компоненте связности либо с ленточкой $a$, либо с ленточкой $c$. Значит, все три ленточки лежат в одной компоненте связности, тогда кривые не разбивают новый диск с ленточками.

Пусть в диске с ленточками, изображенном на рис. \ref{f:ababcdcd}, некоторые ленточки заменены на перекрученные. Очевидно, ленточка $c$ находится в одной компоненте связности дополнения 
%%диска с ленточками до кривых. 
кривых до диска с ленточками. 
Тогда области $I$ и $II$ лежат в этой компоненте. 
%%Тогда 
Следовательно,
%%ленточка 
в ней целиком лежит ленточка
$a$.
% целиком лежит в этой компоненте. 
Значит, все области $I,II,III,IV$ лежат в этой компоненте. Тогда 
%%ленточка 
в ней лежит и ленточка
$b$. 
%лежит в этой компоненте связности. 
Значит, кривые не разбивают диск с ленточками.

\begin{figure}[h]
    \center{\includegraphics[scale=0.5, width=250pt]{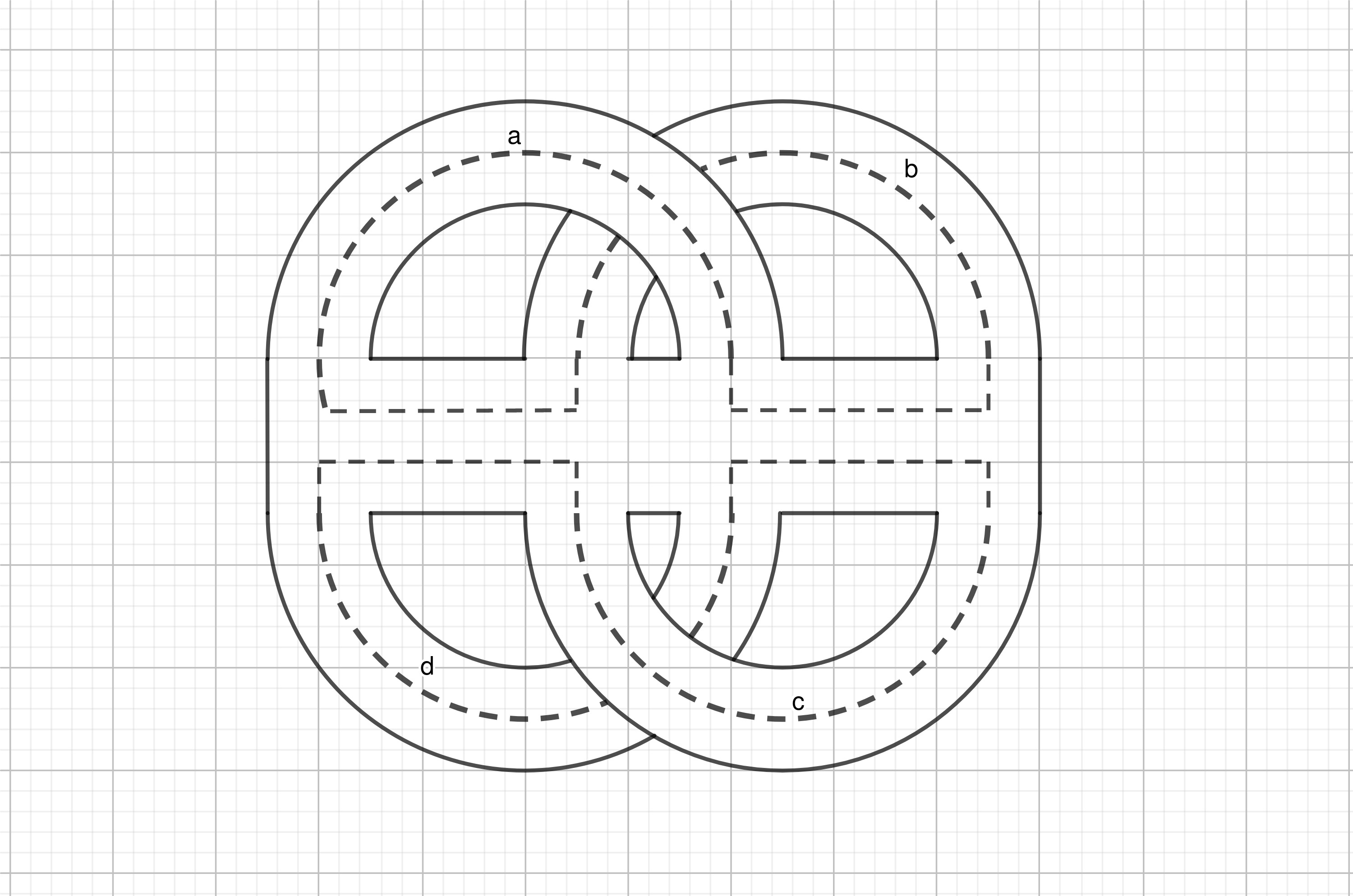}}
    \caption{Диск с ленточками, соответствующий иероглифу $ababcdcd$, и две кривые на нем.}
\label{f:ababcdcd}
\end{figure}

%Пусть в диске с ленточками, изображенном на рис. 5, некоторые ленточки заменены на перекрученные. 
Для иероглифа $ababcdcd$ и диска с ленточками, изображённого на рис. \ref{f:ababcdcd}
%%: если 
, выполнено следующее. Если
заменить ленточку $b$ на перекрученную, а ленточку $a$ не заменять, то каждая из частей ленточки $a$, на которые её разбивает кривая, лежит в одной компоненте связности с центром 
%%диска, т.е. 
диска. Таким образом,
ленточка $a$ целиком лежит в одной компоненте связности, тогда и ленточка $b$ тоже (каждая её часть не отделена от какой-то из частей ленточки $a$). 
%%То же, 
То же верно,
если заменить ленточку $a$ на перекрученную, а ленточку $b$ не заменять. Если же обе ленточки $a$ и $b$ заменить на перекрученные, то связность обеих ленточек можно проверить непосредственно. Для ленточек $c$ и $d$ можно провести полностью аналогичное рассуждение.

\end{lemmaproof}

%\newpage

\bigskip

%\address{
\noindent Бикеев Артур Игоревич, Московский физико-технический институт \\
%(государственный университет)}
%\email{
bikeev99@mail.ru
%}


\begin{thebibliography}{}

\bibitem[VGDNPS]{VGDNPS}\emph{А. Воропаев, Т. Гараев, С. Дженжер, О. Никитенко, А. Петухов, А. Скопенков}. Минимизация ранга восполнением матриц.
\url{https://www.mccme.ru/circles/oim/netflix_rus.pdf}

\bibitem[BFM90]{BFM90} \emph{Болсинов А.В. , Матвеев С.В.,   Фоменко А.Т.}. Топологическая классификация интегрируемых гамильтоновых систем с двумя степенями свободы // УМН. Т. 45, вып. 2 (1990), 49--77.

\bibitem[Ko20]{Ko20} \emph{Kogan E.}. On the minimal rank of a matrix with coefficients in $\mathbb Z_2$ without the numbers on the main diagonal.
\url{http://arxiv.org/abs/2104.10668}

\bibitem[LZ]{LZ}
%% \emph{S.Lando, A.Zvonkin.} Graphs on Surfaces and Their Applications. Springer, 2004.
\emph{Ландо С.К., Звонкин А.К.} Графы на поверхностях и их приложения. М.: МЦНМО. 2010.

\bibitem[Mo89]{Mo89} \emph{Mohar B.}. An obstruction to embedding graphs in surfaces. Discrete Math. V.~78. P.~135"--~142 (1989).

\bibitem[MT01]{MT01} \emph{Mohar B., Thomassen C.}. Graphs on Surfaces. Baltimor, MD: Johns Hopkins University Press. 2001.

\bibitem[Sk20]{Sk20} \emph{Скопенков А.} Алгебраическая топология с геометрической точки зрения. М.: МЦНМО. 2020 
%(2ое издание). \url{http://www.mccme.ru/circles/oim/obstruct.pdf}.

%\bibitem[F]{F} \emph{Фукс Д.} Лента Мёбиуса. Вариации на старую тему // Квант.  1979, № 1. С. 2-9.

\end{thebibliography}
\end{document}